\documentclass[12pt]{amsart}
\usepackage{epsfig,color}
\usepackage{blindtext}
\usepackage{graphicx}
\usepackage{enumitem}
\usepackage{url}
\usepackage{amssymb}
\usepackage{graphicx,import}
\usepackage{comment}
\usepackage{esint}
\usepackage{xcolor}
\usepackage{mathtools}
\usepackage{comment}
\usepackage{tikz}

\usepackage[margin = 1in] {geometry}

\usepackage{hyperref}
\usepackage{dsfont}

\newtheorem{theorem}{Theorem}[section]

\newtheorem{proposition}[theorem]{Proposition}
\newtheorem{lemma}[theorem]{Lemma}

\theoremstyle{definition}

\newtheorem{definition}[theorem]{Definition}
\newtheorem*{remark*}{Remark}
\newtheorem*{acknowledgements}{Acknowledgements}

\numberwithin{equation}{section}

\newcommand{\R}{\mathbb{R}}
\newcommand{\C}{\mathbb{C}}

\newcommand{\N}{\mathbb{N}}
\newcommand{\Z}{\mathbb{Z}}
\newcommand{\eps}{\varepsilon}
\newcommand{\CP}{\mathbb{C}\mathbb{P}}
\newcommand{\kcw}{\mathrm{K}\text{-}\mathrm{cw}}
\newcommand{\acw}{\hat{\mathrm{A}}\text{-}\mathrm{cw}}
\newcommand{\ahat}{\hat{\mathrm{A}}}

\title{Stable 2--systole bounds in positive scalar curvature}
\author{Douglas Stryker}
\address{Department of Mathematics, Stanford University, Building 380, Stanford, CA 94305, USA}
\email{dstryker@stanford.edu}

\begin{document}

\begin{abstract}
We prove that the stable 2--systole is uniformly bounded on the space of Riemannian metrics with scalar curvature at least one for closed spin 2--essential manifolds, which includes $S^2 \times S^2$, $S^2 \times T^n$, and $\mathbb{C}\mathbb{P}^{2n+1}$.
\end{abstract}

\maketitle

\section{Introduction}
A guiding principle in the study of the metric geometry of positive scalar curvature is that it forces at least two dimensions to be small. We refer the reader to \cite{Gromov:lectures} for a comprehensive survey. In dimension three, this principle has been thoroughly verified (for a non--exhaustive list, see \cite{Katz:3D, BBN:sys, Liokumovich-Maximo:3D, Munteanu-Wang:3D}), but fewer results are known in dimension four and higher.

We investigate an instance of this principle that, while challenging, appears most accessible for existing techniques: the \emph{2--systole} of closed Riemannian manifolds with positive scalar curvature. It is expected that, under reasonable topological assumptions, positive scalar curvature should ensure the existence of a non-contractible immersed surface of small area. This expectation has been confirmed for closed manifolds in the following cases:
\begin{itemize}
\item 3--manifolds: \cite{BBN:sys} showed that every closed 3--manifold with scalar curvature at least 2 and $\pi_2 \neq 0$ contains a non-contractible immersed 2--sphere with area at most $4\pi$, in addition to a rigidity result in the case of equality.\footnote{\cite{Stern:intro} gave another proof using harmonic maps to $S^1$. We also mention that \cite{Xu:intro} proved a corresponding gap result using the inverse mean curvature flow.}
\item $S^2 \times T^n$: \cite{Zhu:txs2} proved that every metric with scalar curvature at least 2 on $S^2 \times T^n$ for $n+2 \leq 7$ contains a non-contractible immersed 2--sphere with area at most $4\pi$, with a corresponding rigidity result in the case of equality.
\item $S^2 \times S^2$: \cite{Richard:s2xs2} proved that every ``stretched enough'' metric with scalar curvature at least 4 on $S^2 \times S^2$ contains a non-contractible immersed 2--sphere with bounded area (the area bound becomes asymptotically sharp as the ``stretching'' goes to infinity).
\item K{\"a}hler complex surfaces: \cite{Sha:main} proved that every K{\"a}hler metric on a closed 2--dimensional complex manifold (real dimension 4) with scalar curvature at least 24 contains a nontrivial homologically area--minimizing surface with area at most $\pi$, with equality if and only if the metric is the Fubini--Study metric on $\CP^2$.
\end{itemize}

We emphasize that, in dimension at least four, a uniform bound for the 2--systole on the space of metrics with scalar curvature at least one is only known for $S^2 \times T^n$ for $n+2 \leq 7$.\footnote{The results of \cite{Zhu:txs2} also apply to manifolds that admit a nonzero degree map to $S^2 \times T^n$.} Nonetheless, it is reasonable to expect that such bounds hold for many more closed manifolds, including $S^2 \times T^n$ for $n+2 > 7$, $S^2 \times S^2$, $(S^2)^n$, $(S^2)^m \times T^n$, $\CP^n$, and many more.

In this paper, we use spin methods to prove bounds on the \emph{stable 2--systole} in positive scalar curvature. The stable 2--systole is a well--studied geometric invariant, slightly weakening the concept of 2--systole discussed above. This systolic invariant appears in Gromov's systolic inequality for $\CP^n$ \cite{Gromov:sys0, Gromov:sys1, Gromov:sys2}, as well as subsequent work on curvature--independent systolic bounds (for a non--exhaustive list, see \cite{Hebda:86, Bangert-Katz:sys, GHK:systole, Hebda:sys}). One reason for the utility of the stable 2--systole is that it circumvents systolic freedom---for instance, the existence of unit volume metrics on $S^2 \times S^2$ with arbitrarily large 2--systole due to \cite{Katz-Suciu:freedom}, which does not happen for the stable 2--systole.

We recall the relevant definitions. Let $(M, g)$ be a closed Riemannian manifold. The \emph{2--systole}, denoted $\mathrm{sys}_2$, is the smallest area of a non-contractible immersed $S^2$ in $M$. The \emph{homological 2--systole}, denoted $\mathrm{hom}\text{-}\mathrm{sys}_2$, is the smallest area of a nontrivial homological area minimizing surface in $M$. The \emph{stable 2--systole}, denoted $\mathrm{stsys}_2$, can be defined to be
\begin{equation}\label{eqn:intro-stsys}
\mathrm{stsys}_2 \coloneq \inf\left\{\frac{\mathrm{Area}(\Sigma)}{k} : \Sigma \in k\sigma,\ k \in \Z_{>0},\ \sigma \in H_2(M, \Z),\ \sigma_\R \neq 0 \in H_2(M, \R)\right\}.
\end{equation}
We can think of the stable 2--systole as the homological 2--systole up to ``stabilization,'' where we are allowed to take surfaces in a large integer multiple of a given homology class and then divide the area by that integer.

We mention that \cite{Orikasa:stsys2} used spin techniques to bound the stable relative 2--systole on some disk bundles for metrics with positive scalar curvature satisfying a stretching assumption similar to \cite{Richard:s2xs2} and an admissibility condition at the boundary.

\subsection{Main results}
We recall that a closed manifold $M$ of dimension $2m$ is \emph{2--essential} if a nonzero element in $H^{2m}_{\mathrm{dR}}(M)$ can be written as the product of elements in $H^2_{\mathrm{dR}}(M)$ (see \cite{GHK:systole}). We also recall that a closed manifold $N$ of dimension $2k$ is \emph{enlargeable} if for any Riemannian metric $g$ on $N$ and any $\eps > 0$, there is a finite covering space that admits a nonzero degree map to the unit sphere $\mathbb{S}^{2k}$ with Lipschitz constant at most $\eps$ (see \cite{Gromov-Lawson:closed}).

\begin{theorem}\label{thm:gen}
Let $M$ be a $2m$--dimensional closed 2--essential manifold. Let $N$ be a $2k$--dimensional closed enlargeable manifold, or a point. Suppose $M \times N$ is spin, and let $b = b_2(M \times N)$ be the second Betti number. Let $g$ be a Riemannian metric on $M \times N$ with $\mathrm{scal}(g) \geq 1$. Then
\[ \mathrm{stsys}_2(g) \leq C_{b, k, m}. \]
\end{theorem}

The 2--essential condition is required for Theorem \ref{thm:gen}. For example, $S^2 \times S^4$ is not 2--essential, and it admits metrics with scalar curvature at least one and arbitrarily large stable 2--systole (for instance, $(\eps^{-1}g_{\mathbb{S}^2}) \oplus g_{\mathbb{S}^4}$ for $\eps \to 0$).

The core of Theorem \ref{thm:gen} is the case when $N$ is a point (the 2--essential case). The case when $N$ is not a point relates to the 2--essential case in the same way that the submersion/product results of \cite{Llarull:sub, Goette-Semmelmann, Tony, Riedler-Tony} relates to the result of \cite{Llarull}.

We now provide several concrete applications of our theory.

\subsubsection{Products of 2--spheres}
A particularly challenging instance of the (stable) 2--systole problem is $S^2 \times S^2$. Unlike $S^2 \times T^n$, where the enlargeability of the torus factor can be exploited, $S^2 \times S^2$ is simply connected. As mentioned above, \cite{Richard:s2xs2} proved an upper bound for the 2--systole under the assumption that the metric is sufficiently ``stretched.'' \cite{Sha:main} proved the sharp bound for the homological 2--systole in the K{\"a}hler setting. It is natural to ask whether the 2--systole is bounded without the stretching or K{\"a}hler assumption (see \cite[Question 6.2]{Richard-Zhu:survey}).

We prove the following bound for the stable 2--systole on $S^2 \times S^2$.

\begin{theorem}\label{thm:s2xs2}
If $(S^2 \times S^2, g)$ satisfies $\mathrm{scal}(g) \geq 4$, then $\mathrm{stsys}_2(g) \leq 36\pi$.
\end{theorem}

The expected bound is $4\pi$, so our result fails to be sharp by a factor of $9$. Beyond $S^2 \times S^2$, our method applies to the product of any finite number of copies of $S^2$.

\begin{theorem}\label{thm:s2n}
If $((S^2)^n, g)$ satisfies $\mathrm{scal}(g) \geq 2n$, then $\mathrm{stsys}_2(g) \leq C_n = O(n^4\log n)$.
\end{theorem}

In \S\ref{sec:s2}, we make a digression to show that our analysis is sharp on $S^2$, and we give a spin proof of the sharp area bound and corresponding rigidity statement for metrics on $S^2$ with Gaussian curvature at least one.

\subsubsection{Products of 2--spheres and a torus}\label{sec:txs2}
\cite{Zhu:txs2} proved the sharp 2--systole upper bound for $S^2 \times T^n$ when $n + 2 \leq 7$ using geometric measure theory. It is natural to ask whether the (stable) 2--systole is bounded when $n+2 > 7$ (see \cite[Conjecture 6.1]{Richard-Zhu:survey}).

Before stating our result, we make an important remark about the case of $S^2 \times T^n$. In general, the stable 2--systole may be achieved for the homology class of a 2--torus inside the $T^n$ factor. Since $T^n$ is enlargeable, we expect to be able to bound the (stabilized) norm of the homology class of $S^2 \times \{\mathrm{pt}\}$. For this reason, we introduce a strengthened version of the stable 2--systole, called the \emph{stable spherical 2--systole} and denoted $\mathrm{stsys}_2^{\mathrm{sph}}$. The stable spherical 2--systole is the same as the stable 2--systole, except we only allow classes $\sigma$ in \eqref{eqn:intro-stsys} that can be represented by spheres. The precise definition can be found in \S\ref{sec:def-sys}. We note that $\mathrm{stsys}_2^{\mathrm{sph}} \geq \mathrm{stsys}_2$.

\begin{theorem}\label{thm:txs2}
If $(S^2 \times T^n, g)$ satisfies $\mathrm{scal}(g) \geq 1$, then $\mathrm{stsys}_2^{\mathrm{sph}}(g) \leq C_n$.
\end{theorem}

While the constant $C_n$ is computable, it is certainly not sharp. Beyond $S^2 \times T^n$, our method applies to the product of multiple copies of $S^2$ with a torus.

\begin{theorem}\label{thm:txs2n}
If $((S^2)^m \times T^n, g)$ satisfies $\mathrm{scal}(g) \geq 1$, then $\mathrm{stsys}_2^{\mathrm{sph}}(g) \leq C_{m,n}$.
\end{theorem}

\subsubsection{Complex projective space}
Gromov showed that for any Riemannian $(\CP^n, g)$, the following sharp inequality holds (see \cite{Gromov:sys0, Gromov:sys1, Gromov:sys2}):
\[ \mathrm{stsys}_2(g)^n \leq n!\mathrm{Vol}(\CP^n, g). \]
Hence, if $\mathrm{Ric}(g) \geq 1$, the Bishop--Gromov volume comparison (see \cite[Lemma 7.1.4]{Petersen:geometry}) combines with the systolic inequality to give a uniform upper bound on the stable 2--systole. It is natural to ask whether a uniform bound holds under the weaker assumption that the scalar curvature is positive (see \cite[Question 6.3]{Richard-Zhu:survey}). \cite{Sha:main} proved the sharp bound on the homological 2--systole for K{\"a}hler metrics on $\CP^2$. We show:

\begin{theorem}\label{thm:cpn}
If $(\CP^{2n+1}, g)$ satisfies $\mathrm{scal}(g) \geq 1$, then $\mathrm{stsys}_2(g) \leq C_n$. 
\end{theorem}

As a digression, we observe a curious rigidity phenomenon for K{\"a}hler metrics with positive scalar curvature on $\CP^{2n+1}$ (which may already be well--known to K{\"a}hler geometry experts)---they have bounded volume.

\begin{theorem}\label{thm:cpn-kahler}
If $(\CP^{2n+1}, g)$ is K{\"a}hler and $\mathrm{scal}(g) \geq 1$, then $\mathrm{Vol}(\CP^{2n+1}, g) \leq C_n$.
\end{theorem}

For concreteness, we compute the explicit bounds for $\CP^3$.

\begin{theorem}\label{thm:cp3}
If $(\CP^3, g)$ satisfies $\mathrm{scal}(g) \geq 48 = \mathrm{scal}(g_{\mathrm{FS}})$, then $\mathrm{stsys}_2(g) \leq 5\pi$. If $g$ is K{\"a}hler, then $\mathrm{Vol}(\CP^3, g) \leq \frac{125\pi^3}{6}$.
\end{theorem}

The Fubini--Study metric on $\CP^3$ is expected to be extremal for the systole and volume bounds, so the expected sharp bounds are $\pi$ for the stable 2--systole and $\frac{\pi^3}{6}$ for the volume.

It seems reasonable to expect that the same results hold for $\CP^{2n}$, but our techniques do not directly apply because $\CP^n$ is spin if and only if $n$ is odd.

\subsubsection{Non 2--essential example}
\cite{GHK:systole} (see also \cite{Gromov:sys0, Bangert-Katz:sys}) proved a bound on the stable 2--systole in terms of the volume for any closed 2--essential manifold. For closed manifolds that are not 2--essential, the only known general bounds for the stable 2--systole require control on the volume in addition to the stable $k$--systole for some $k \neq 2$ (see \cite{Gromov:sys0, Bangert-Katz:sys, Hebda:sys}). We exhibit a closed manifold that is not 2--essential for which we can prove a uniform bound on the stable 2--systole for metrics with scalar curvature at least one.

\begin{theorem}\label{thm:non-2-ess}
Let $X$ be a hyperbolic homology 3--sphere\footnote{See for example \cite{hyp} and the references therein for a proof of the existence of such a manifold.}. If $(S^2 \times X\times X, g)$ satifsies $\mathrm{scal}(g) \geq 1$, then $\mathrm{stsys}_2(g) \leq C$.
\end{theorem}

\subsection{Strategy and discussion}
In \cite{Gromov:K-area}, Gromov introduced the notion of cowaist (called $\mathrm{K}$--area in \cite{Gromov:K-area}) and showed that the twisted spinor techniques of \cite{Gromov-Lawson:closed} give upper bounds for the cowaist when the scalar curvature is positive. Roughly speaking, the cowaist is the inverse of the smallest possible curvature over all sufficiently topologically nontrivial Hermitian vector bundles. In \cite{Gromov:K-area} and, as far as the author is aware, almost all subsequent related works (see the non--exhaustive list \cite{Davaux, Listing:intro, Bar-Hanke:intro, Wang:cowaist, Cecchini-Zeidler:spin, Shi:intro, Su-Wang:intro}), the cowaist bound is used in situations where the cowaist of the manifold (or the cowaist of a suitable submanifold) is shown or assumed to be infinite. Typically, arguments showing that the cowaist is infinite follow coarse estimates for the curvature of Hermitian vector bundles, usually given by pulling back a model Hermitian bundle on a model manifold by a map with controlled Lipschitz constant. In these cases, precise quantitative control of the curvature of any given bundle is not required.

A notable exception is the theorem of Llarull \cite{Llarull}, where the precise value of the curvature of the spinor bundle on the round sphere is used to obtain a sharp result. This approach was generalized by \cite{Goette-Semmelmann, Tony} to manifolds with nonnegative curvature operator and nonzero Euler characteristic. However, there seems to be no general setting where a well--chosen finite dimensional Hermitian bundle with precisely controlled curvature on an arbitrary metric with positive scalar curvature has been used to deduce geometric conclusions. 

In this paper, we take on this task by studying the curvature of Hermitian \emph{line} bundles, which can be computed precisely using Chern--Weil theory. In particular, we can control the curvature of a well--chosen line bundle by a certain invariant of the integer lattice in the Banach space $H^2_{\mathrm{dR}}(M)$ with the \emph{comass} norm. Using the theory of lattices in Banach spaces, we can then control the stable 2--systole, which is the shortest length vector in the dual lattice in $H_2(M, \R)$ with the \emph{stable} norm.

We emphasize that our analysis is relatively elementary and quite classical, and yet it still provides control on the stable 2--systole by the cowaist. An important question left open by this paper is how much better one can do by looking at higher rank Hermitian vector bundles. This question seems daunting, however, as we are not aware of any general higher rank situation where the full curvature can be prescribed or controlled.

The main disadvantages of the spin techniques developed here, as compared to the geometric measure theory methods of \cite{BBN:sys, Zhu:txs2, Richard:s2xs2}, are
\begin{enumerate}
\item the manifold must be spin,
\item they apply to the weaker \emph{stable} 2--systole invariant, and
\item they do not give the expected sharp bound (except in dimension 2).
\end{enumerate}
However, these techniques have a few important advantages, including
\begin{enumerate}
\item they apply in all dimensions, and
\item they apply to a large general class of manifolds.
\end{enumerate}

\subsection{Organization}
In \S\ref{sec:norms}, we recall the relevant norms on homology and cohomology, and introduce the stable 2--systole. In \S\ref{sec:spin}, we define the cowaist and recall the cowaist inequality of \cite{Gromov:K-area}. In \S\ref{sec:chern-weil}, we recall the Chern--Weil theory for line bundles required for our analysis. In \S\ref{sec:s2}, we focus on the case of $S^2$. In \S\ref{sec:cw}, we carry out the main analysis, bounding the cowaist by the stable 2--systole. In \S\ref{sec:main}, we combine all of the preceding work to prove the main results.

\begin{acknowledgements}
The author is grateful to Thomas Tony for his interest and feedback on this work, and to Otis Chodosh, Filippo Gaia, and Thomas Massoni for insightful discussions. The author was supported by the NSF grant DMS-2503279.
\end{acknowledgements}

\section{Norms on homology and cohomology}\label{sec:norms}

We recall the relevant norms on homology and cohomology, as established by \cite{Federer:norms}. The following subsections draw on the exposition of \cite[Chapter 4, \S C]{Gromov:sys2}, \cite[\S3]{GHK:systole}, and \cite[\S2]{Orikasa:stsys2}). Thoughout, we let $(M, g)$ be a closed Riemannian manifold.

\subsection{Norms on 2--forms}
For any 2--form $\omega$ on $M$, we define
\[ \|\omega\|_{\infty} \coloneq \sup\{|\omega_p(X, Y)| : p\in M, \ X,Y \in T_pM,\ |X \wedge Y| = 1\}. \]

For $\Xi \in H^2_{\mathrm{dR}}(M)$, we define the \emph{comass norm} of $\Xi$ to be
\[ \|\Xi\|^* \coloneq \inf\{\|\omega\|_{\infty} : \omega \in \Xi\}. \]
By \cite{Federer:norms}, $(H^2_{\mathrm{dR}}(M), \|\cdot\|^*)$ is a Banach space.

For any Hermitian vector space $V$ and any endomorphism $A : V \to V$, we define the operator norm
\[ |A|_{\mathrm{op}} \coloneq \sup\{|A(v)| : v \in V,\ |v| = 1\}. \]

For any Hermitian vector bundle $E \to M$ and any $\mathrm{End}(E)$--valued 2--form $\mathcal{R}$, we define
\[ \|\mathcal{R}\|_{\infty} \coloneq \sup\{|\omega_p(X,Y)|_{\mathrm{op}} : p\in M, \ X, Y \in T_pM,\ |X \wedge Y| = 1\}. \]

\subsection{Norms on 2--cycles}
For any integer homology class $\sigma \in H_2(M, \Z)$, the \emph{mass norm} of $\sigma$ is
\[ \mathbb{M}(\sigma) \coloneq \inf\{\mathrm{Area}(\Sigma) : \Sigma \ \text{integer\ cycle\ in\ } \sigma\}. \]
In other words, $\mathbb{M}(\sigma)$ is the area of any homological area minimizer in the class $\sigma$.

For any real homology class $\sigma \in H_2(M, \R)$, the \emph{norm} of $\sigma$ is
\[ \|\sigma\| \coloneq \inf\{\mathrm{Area}(\Sigma) : \Sigma \ \text{real\ cycle\ in\ } \sigma\}. \]
By \cite{Federer:norms}, $(H_2(M, \R), \|\cdot\|)$ is a Banach space.

For $\sigma \in H_2(M, \Z)$ the \emph{stable norm} of $\sigma$ is
\[ \|\sigma\| \coloneq \|\sigma_\R\|, \]
where $\sigma_\R$ is the image of $\sigma$ in $H_2(M, \R)$.

Let $\sigma \in H_2(M, \Z)$. By \cite{Federer:norms}, we have
\[ \|\sigma\| = \lim_{k \to \infty} \frac{1}{k}\mathbb{M}(k\sigma). \]
In other words, there is a sequence of integers $\{k_j\}_{j \in \N}$ and area minimizing integer cycles $\{\Sigma_j \in k_j\sigma\}_{j \in \N}$ so that
\[ \lim_{j \to \infty} \frac{\mathrm{Area}(\Sigma_j)}{k_j} = \|\sigma\|. \]
It is for this reason that the norm is referred to as the stable norm, since it is the mass norm up to stabilizing by taking integer multiples.

\subsection{Lattices and duality}
Let $\Lambda_2(M)$ denote the image of $H_2(M, \Z)$ in $H_2(M, \R)$. Note that $\Lambda_2(M)$ is a lattice in $H_2(M,\R)$.

Let $\Lambda^2(M)$ denote the image of $H^2(M, \Z)$ in $H^2_{\mathrm{dR}}(M)$. Note that $\Lambda^2(M)$ is a lattice in $H^2_{\mathrm{dR}}(M)$.

By \cite{Federer:norms}, $(H^2_{\mathrm{dR}}(M), \|\cdot\|^*)$ and $(H_2(M, \R), \|\cdot\|)$ are dual Banach spaces, and $\Lambda^2(M)$ and $\Lambda_2(M)$ are dual lattices.

For any lattice $\Lambda$ in a Banach space $(V, |\cdot|)$, let
\[ \lambda_k(\Lambda) \coloneq \inf\left\{\max_{j \in \{1, \hdots, k\}} |v_j| : v_1, \hdots, v_k\ \text{are\ linearly\ independent\ in\ } \Lambda\right\}. \]
As in \cite[\S2]{GHK:systole}, we define $\Gamma_b$ to be the supremum of $\lambda_1(\Lambda)\lambda_b(\Lambda^*)$ over all lattices in any $b$--dimensional Banach space, where $\Lambda^*$ is the dual lattice to $\Lambda$. It is clear that $\Gamma_1 = 1$. By an argument of David Speyer (see \cite[Proposition 2.3]{GHK:systole}), $\Gamma_2 = \frac{3}{2}$. By \cite{Banaszczyk:gamma} (see \cite[\S2]{GHK:systole}), we have $\Gamma_b = O(b \log b)$.

\subsection{Stable 2--systole}\label{sec:def-sys}
The \emph{2--systole} is
\[ \mathrm{sys}_2(g) \coloneq \inf\{\mathrm{Area}(S) : S\ \text{is\ a\ non-contractible\ immersed\ 2--sphere\ in\ } M\}. \]

The \emph{homological 2--systole} is
\[ \mathrm{hom}\text{-}\mathrm{sys}_2(g) \coloneq \inf\{\mathrm{Area}(\Sigma) : \Sigma\ \text{is\ a\ homologically\ non-zero\ integer\ 2--cycle\ in\ } M\}. \]
In other words, the homological 2--systole is the area of the least area nontrivial area minimizing surface in $(M, g)$.

The \emph{stable 2--systole} is
\[ \mathrm{stsys}_2(g) \coloneq \lambda_1(\Lambda_2(M)). \]
In other words, the stable 2--systole is the stable norm of the smallest nonzero element in the lattice $\Lambda_2(M)$. It represents the area of an area minimizer up to stabilization by taking integer multiples.

For $(S^2)^m \times T^n$, we make a new specialized definition. For $j \in \{1, \hdots, m\}$, let $\sigma_j$ be the homology class of points times the 2--sphere in the $j$--th factor:
\[ \{\mathrm{pt}\} \times \cdots \times S^2 \times \cdots \times \{\mathrm{pt}\} \times \{\mathrm{pt}\} \subset (S^2)^m \times T^n. \]
We define the \emph{stable spherical 2--systole} to be
\[ \mathrm{stsys}_2^{\mathrm{sph}}(g) \coloneq \inf\{\|\sigma\| : \sigma \in \mathrm{Span}(\sigma_1, \hdots, \sigma_m) \cap \Lambda_2(M) \setminus \{0\}\}. \]

\section{Spinors and cowaist}\label{sec:spin}
In this section, we recall two definitions of the cowaist of a Riemannian manifold, as well as the upper bound on the cowaist in positive scalar curvature coming from the index theory of twisted spinor bundles pioneered in \cite{Gromov-Lawson:closed}.

\subsection{K--cowaist}
For a complex vector bundle $E \to M$, we let $c_j(E)$ denote the $j$--th Chern class of $E$. We define the $\mathrm{K}$--cowaist  following \cite{Gromov:K-area} (where it was originally called \emph{$\mathrm{K}$--area}).

\begin{definition}
Let $(M,g)$ be a closed Riemannian manifold. Let $E \to M$ be a Hermitian vector bundle with metric connection, and let $R^E$ be the curvature 2--form of the connection. We say $E$ is \emph{$\mathrm{K}$--admissible} if $E$ has a nonvanishing characteristic Chern number, meaning that there are $i_1,\hdots, i_l \in \Z_{>0}$ so that
\[ \int_M \prod_{j=1}^l c_{i_j}(E) \neq 0. \]
The \emph{$\mathrm{K}$--cowaist} of $M$ is
\[ \kcw(g) \coloneq \left(\inf\left\{\|R^E\|_{\infty} : E\ \text{is}\ \mathrm{K}\text{--admissible}\right\}\right)^{-1} \]
\end{definition}

The following inequality follows from the techniques of \cite{Gromov-Lawson:closed}, and in this form is due to \cite[Theorem 5$\frac{1}{4}$]{Gromov:K-area} (see also \cite[\S1]{Wang:cowaist}).

\begin{theorem}\label{thm:spin}
There are constant $c_n > 0$ so that the following holds. Let $(M, g)$ be a closed Riemannian spin manifold of dimension $2n$. Then
\[ \frac{c_n}{\min_M \mathrm{scal}(g)} \geq \kcw(g). \]
\end{theorem}

\subsection{A-hat--cowaist}
We now make some refined definitions for better quantitative control. We point the reader to \cite[Chapter III, Example 11.13]{Lawson-Michelson:spin} for the definition of the total $\hat{\mathbf{A}}$--class of a vector bundle (here we use $\hat{\mathbf{A}}(M) \coloneq \hat{\mathbf{A}}(TM)$), and to \cite[Chapter III, (11.22)]{Lawson-Michelson:spin} for the definition of the Chern character $\mathrm{ch}(E)$ of a complex bundle $E$. We define the $\ahat$--cowaist following for example \cite{Wang:cowaist}, and we then make a refined definition by restricting to Hermitian line bundles.

\begin{definition}
Let $(M,g)$ be a closed Riemannian manifold. Let $E \to M$ be a Hermitian vector bundle with metric connection, and let $R^E$ be the curvature 2--form. We say $E$ is \emph{$\ahat$--admissible} if
\[ \int_M \hat{\mathbf{A}}(M) \cdot \mathrm{ch}(E) \neq 0. \]
The \emph{$\ahat$--cowaist} of $M$ is
\[ \acw(g) \coloneq \left(\inf\left\{\|R^E\|_{\infty} : E\ \text{is}\ \ahat\text{--admissible}\right\}\right)^{-1} \]
The \emph{line bundle $\ahat$--cowaist} of $M$ is
\[ \acw^{\mathrm{line}}(g) \coloneq \left(\inf\left\{\|R^L\|_{\infty} : L\ \text{is\ an}\ \ahat\text{--admissible\ line\ bundle}\right\}\right)^{-1} \]
\end{definition}

Note that by definition we have $\acw^{\mathrm{line}}(g) \leq \acw(g)$.

The following inequality is essentially proved in the course of the proof of \cite[Theorem 5$\frac{1}{4}$]{Gromov:K-area} (see also \cite[\S1]{Wang:cowaist}).

\begin{theorem}\label{thm:spin-quant}
Let $(M, g)$ be a closed Riemannian spin manifold of dimension $2n$. Then
\[ \frac{4n(2n-1)}{\min_M\mathrm{scal}(g)} \geq \acw^{\mathrm{line}}(g). \]
\end{theorem}
\begin{proof}
We need only compute the constant depending on dimension in the proof of \cite[Theorem 5$\frac{1}{4}$]{Gromov:K-area}.

By inspecting the proof of \cite[Theorem 5$\frac{1}{4}$]{Gromov:K-area}, it suffices to show that
\begin{equation}\label{eqn:line} |\mathcal{R}^L(\psi)| \leq n(2n-1)\|R^L\|_{\infty}|\psi|, \end{equation}
for any Hermitian line bundle with metric connection $L \to M$, where $\mathcal{R}^L$ is defined in \cite[Chapter II, (8.22)]{Lawson-Michelson:spin}, and $R^L$ is the curvature 2--form of the connection on $L$.

Fix $x \in M$. Let $s_1, \hdots, s_r$ be an orthonormal basis for the complex spinor bundle $\mathbf{S}_\C \to M$ at $x$, let $\xi$ be a nonzero unit length element in the fiber of $L$ at $x$, and let $e_1, \hdots, e_{2n}$ be an orthonormal basis for $T_xM$. Then
\[ \psi(x) = \left(\sum_{i=1}^r a_is_i\right) \otimes \xi, \]
which is a simple element of $\mathbf{S}_\C\otimes L$ with
\[ |\psi(x)|^2 = \sum_i a_i^2. \]
Then we compute
\begin{align*}
|\mathcal{R}^L(\psi)|(x)
& = \left|\frac{1}{2}\sum_{p,q} e_p\cdot e_q \cdot \left(\sum_i a_is_i\right) \otimes R^L_{e_p, e_q}\xi \right|\\
& \leq \sum_{p < q} \left|\sum_i a_is_i\right|\|R^L\|_{\infty} \\
& = n(2n-1)\sqrt{\sum_i a_i^2}\|R^L\|_{\infty}\\
& = n(2n-1)\|R^L\|_{\infty}|\psi|(x),
\end{align*}
as desired.
\end{proof}

\section{Curvature of complex line bundles}\label{sec:chern-weil}

The following result, which plays a central role in our analysis, is a well--known consequence of Chern--Weil theory. We refer the reader to \cite[\S1.8]{Moore:SW} for a proof.

\begin{lemma}\label{lem:prescribe-curvature}
Let $M$ be a closed manifold of dimension at least 2. Let $\omega$ be a closed 2--form on $M$ so that $[\omega]$ is an integer cohomology class. Then there is a Hermitian line bundle $L_{\omega} \to M$ with a metric connection satisfying
\[ R^{L_{\omega}} = -2\pi i\omega\ \ \text{and}\ \ c_1(L_{\omega}) = [\omega]. \]
\end{lemma}

\section{Digression: Cowaist of the 2--sphere}\label{sec:s2}
We make a brief digression to provide sharp analysis in the case of the 2--sphere.

\begin{proposition}\label{prop:s2}
Let $(S^2, g)$ be a Riemannian 2--sphere. Then
\[ \acw^{\mathrm{line}}(g) \geq \frac{1}{2\pi}\mathrm{Area}(S^2, g). \]
\end{proposition}
\begin{proof}
Let $\Xi$ be the positive generator of $H^2(S^2, \Z)$. Let $\omega$ be the area 2--form of the metric $g$. Note that $\|\omega\|_{\infty} = 1$. We have $[\omega] = A\Xi$ for some $A > 0$. Note that
\[ \mathrm{Area}(S^2, g) = \int_{S^2} \omega = A\int_{S^2}\Xi = A. \]
Let $L$ be the Hermitian line bundle with metric connection associated to $\frac{\omega}{\mathrm{Area}(S^2, g)} \in \Xi$ from Lemma \ref{lem:prescribe-curvature}. Then we have
\[ c_1(L) = \Xi, \]
which implies that $L$ is $\ahat$--admissible, since $\hat{\mathbf{A}}(S^2) = 1$. Moreover, we have
\begin{equation}\label{eqn:s2} \|R^L\|_{\infty} = \frac{2\pi\|\omega\|_{\infty}}{\mathrm{Area}(S^2, g)} = \frac{2\pi}{\mathrm{Area}(S^2, g)}, \end{equation}
so the conclusion follows.
\end{proof}

By the above arguments, we can give a spinor proof of the following classical consequence of the Gauss--Bonnet formula.

\begin{theorem}
If $(S^2, g)$ has $K_g \geq 1$, then $\mathrm{Area}(S^2, g) \leq 4\pi$, with equality if and only if $(S^2, g)$ is isometric to the unit sphere $\mathbb{S}^2$.
\end{theorem}
\begin{proof}
Let $L$ be the line bundle as in the proof of Proposition \ref{prop:s2}. By the Atiyah--Singer index theorem (see \cite[Chapter III, Theorem 13.10]{Lawson-Michelson:spin}), since $L$ is $\ahat$--admissible, the Dirac operator $\mathbf{D}_L$ on the twisted spinor bundle $\mathbf{S}_{\C} \otimes L$ constructed in \cite[Chapter II, Proposition 5.10]{Lawson-Michelson:spin} admits a nonzero section $\psi$ of $\mathbf{S}_{\C} \otimes L$ in the kernel of $\mathbf{D}_L$. By the twisted Lichnerowicz formula (see \cite[Chapter II, Theorem 8.17]{Lawson-Michelson:spin}), \eqref{eqn:line}, and \eqref{eqn:s2}, we have
\[ 0 = \int_{S^2} |\mathbf{D}_L\psi|^2 \geq \int_{S^2} |\nabla \psi|^2 + \frac{1}{2}K_g|\psi|^2 - \frac{2\pi}{\mathrm{Area}(S^2, g)}|\psi|^2 \geq \frac{1}{2}\int_{S^2} \left(1 - \frac{4\pi}{\mathrm{Area}(S^2, g)}\right)|\psi|^2. \]
Since $\psi$ is not identically zero, the inequality follows. Moreover, if $\mathrm{Area}(S^2, g) = 4\pi$, then $|\nabla \psi| \equiv 0$, which implies that $\nabla |\psi| \equiv 0$, meaning that $\psi$ has constant nonzero length. Then we must have $K_g \equiv 1$. The conclusion now follows from the classification of constant curvature metrics on simply connected spaces (see for instance \cite[Corollary 5.6.14]{Petersen:geometry}).
\end{proof}

\section{Cowaist lower bounds}\label{sec:cw}
By making judicious choices of Hermitian bundles with prescribed curvature using Lemma \ref{lem:prescribe-curvature}, we prove lower bounds for the cowaist in terms of the stable 2--systole for a large class of manifolds.

\subsection{K--cowaist of 2--essential manifolds}\label{sec:cw-2-ess}

As in \cite{GHK:systole}, we say a closed manifold of dimension $2m$ is \emph{2--essential} if a nonzero element in $H^{2m}_{\mathrm{dR}}(M)$ can be written as the product of classes in $H^2_{\mathrm{dR}}(M)$. We prove a lower bound for the K--cowaist of 2--essential manifolds.

\begin{lemma}\label{lem:cw-gen}
Let $M$ be a $2m$--dimensional closed 2--essential manifold with second Betti number $b_2(M) = b$. Let $g$ be a Riemannian metric on $M$. Then
\[ \kcw(g) \geq \frac{\mathrm{stsys}_2(g)}{2\pi \Gamma_b}. \]
\end{lemma}
\begin{proof}
Let $\Xi_1, \hdots, \Xi_b$ be a linearly independent spanning set for $\Lambda^2(M)$ with
\[ \|\Xi_j\|^* \leq \lambda_b(\Lambda^2(M)). \]
As in \cite[Proof of Theorem 3.3]{GHK:systole} (without loss of generality by reordering), there are positive integers $k_1, \hdots, k_l \in \Z_{>0}$ with
\[ \sum_{j=1}^l k_j = m \]
so that
\[ \Xi_1^{k_1} \cdots \Xi_l^{k_l} \neq 0 \in H^{2m}_{\mathrm{dR}}(M). \]
Fix $\eps > 0$, and let $\omega_j \in \Xi_j$ so that
\[ \|\omega_j\|_{\infty} \leq \|\Xi_j\|^* + \eps/(2\pi). \]
Let $L_j \to M$ be the Hermitian line bundle with metric connection associated to $\omega_j$ from Lemma \ref{lem:prescribe-curvature}. Then the bundle
\[ E = L_1^{\oplus k_1} \oplus \cdots \oplus L_l^{\oplus k_l} \]
satisfies
\[ c_m(E) = \Xi_1^{k_1}\cdots \Xi_l^{k_l} \neq 0, \]
so $E$ is $\mathrm{K}$--admissible. Moreover, we have by \cite[\S4, Page 18]{Gromov:K-area}
\begin{align*}
\|R^E\|_{\infty}
& \leq \max_{j \in \{1, \hdots, l\}} \|R^{L_j}\|_{\infty}\\
& = \max_{j \in \{1, \hdots, l\}} 2\pi\|\omega_j\|_{\infty}\\
& \leq \max_{j \in \{1, \hdots, l\}} 2\pi\|\Xi_j\|^* + \eps\\
& \leq 2\pi \lambda_b(\Lambda^2(M)) + \eps.
\end{align*}
Since $\eps > 0$ was arbitrary, by \cite[\S2]{GHK:systole}, we have
\[ \kcw(g) \geq \frac{1}{2\pi \lambda_b(\Lambda^2(M))} \geq \frac{\mathrm{stsys}_2(g)}{2\pi \Gamma_b}, \]
as desired.
\end{proof}

The above technique also applies if we take a product with an enlargeable manifold. As in \cite{Gromov-Lawson:closed}, we say that a closed manifold $N$ of dimension $2k$ is \emph{enlargeable} if for any Riemannian metric $g$ on $N$ and any $\eps > 0$, there is a finite covering space that admits a nonzero degree map to the unit sphere $\mathbb{S}^{2k}$ with Lipschitz constant at most $\eps$.

\begin{lemma}\label{lem:cw-gen-prod}
Let $M$ be a $2m$--dimensional closed 2--essential manifold. Let $N$ be a $2k$--dimensional closed enlargeable manifold. Let $b = b_2(M \times N)$ be the second Betti number of $M \times N$, and let $g$ be a Riemannian metric on $M \times N$. Then
\[ \kcw(g) \geq \frac{\mathrm{stsys}_2(g)}{2\pi \Gamma_b}. \]
\end{lemma}
\begin{proof}
Let $\zeta_N$ be the pullback of a generator $\bar{\zeta}_N \in H^{2k}(N, \Z)$ by the projection map
\[ \pi_N: M \times N \to N. \]
Fix a Riemannian metric $g_N$ on $N$. Let
\[ C_N \coloneq \mathrm{Lip}(\pi_N : (M \times N, g) \to (N, g_N)) < \infty. \]
Since $N$ is enlargeable, by \cite{Gromov-Lawson:closed}, there is a Hermitian vector bundle with metric connection $E_N^0 \to N$ with $c_k(E_N^0) = a_1 \bar{\zeta}_N$ for some $a_1 \neq 0$, and
\[ \|R^{E_N^0}\|_{\infty} \leq \frac{\eps}{C_N^2}. \]
If we let $E_N \to M \times N$ be the pullback of $E_N^0$ by $\pi_N$, then $c_k(E_N) = a_1 \zeta_N$ for $a_1 \neq 0$, $c_l(E_n) = 0$ for $l > k$, and
\[ \|R^{E_N}\|_{\infty} \leq \eps. \]

Let $\zeta_M$ be the pullback of a generator $\bar{\zeta}_M \in H^{2m}(M, \Z)$ by the projection map
\[ \pi_M : M \times N \to M. \]
Let $\eta_1, \hdots, \eta_m \in H^2_{\mathrm{dR}}(M)$ so that $\eta_1 \cdots \eta_m = a_2 \bar{\zeta}_M$ for some $a_2 \neq 0$, which exists because $M$ is 2--essential.
Let $\Xi_1, \hdots, \Xi_b$ be a linearly independent spanning set for $\Lambda^2(M \times N)$ with
\[ \|\Xi_j\|^* \leq \lambda_b(\Lambda^2(M\times N)). \]
Writing each $\pi_M^*(\eta_i)$ in the basis $\{\Xi_j\}$ and taking the product, we find that there are positive integers $k_1, \hdots, k_l$ summing to $m$ so that
\[ \Xi_1^{k_1} \cdots \Xi_l^{k_l} = a_3\zeta_M + \beta \]
for some $a_3 \neq 0$ and some $\beta \in H^{2m}_{\mathrm{dR}}(M \times N)$ so that $\zeta_N \cdot \beta = 0$. Let $\omega_j \in \Xi_j$ so that
\[ \|\omega_j\|_{\infty} \leq \|\Xi_j\|^* + \eps/(2\pi). \]
Let $L_j \to M \times N$ be the Hermitian line bundle with metric connection associated to $\omega_j$ from Lemma \ref{lem:prescribe-curvature}. Then the bundle
\[ E_M = L_1^{\oplus k_1} \oplus \cdots \oplus L_l^{\oplus k_l} \]
satisfies
\[ c_m(E_M) = \Xi_1^{k_1}\cdots \Xi_l^{k_l} = a_3\zeta_M + \beta, \]
and $c_l(E_M) = 0$ for all $l > n$. Moreover, by \cite[\S4, Page 18]{Gromov:K-area}, we have
\begin{align*}
\|R^{E_M}\|_{\infty}
& \leq \max_{j \in \{1, \hdots, l\}} \|R^{L_j}\|_{\infty}\\
& = \max_{j \in \{1, \hdots, l\}} 2\pi\|\omega_j\|_{\infty}\\
& \leq \max_{j \in \{1, \hdots, l\}} 2\pi\|\Xi_j\|^* + \eps\\
& \leq 2\pi\lambda_b(\Lambda^2(M \times N)) + \eps.
\end{align*}

We set $E = E_N \oplus E_M$. By the formula for the Chern classes under direct sum, we have
\[ c_{m+k}(E) = a_1a_3\zeta_M\cdot \zeta_N \neq 0 \in H^{2(m+k)}(M \times N), \]
so $E$ is $\mathrm{K}$--admissible. Moreover, by \cite[\S4, Page 18]{Gromov:K-area}, we have
\begin{align*}
\|R^E\|_{\infty}
& \leq \max\left\{\|R^{E_M}\|_{\infty},\ \|R^{E_N}\|_{\infty}\right\} \leq 2\pi\lambda_b(\Lambda^2(M \times N)) + \eps.
\end{align*}
Since $\eps > 0$ was arbitrary, by \cite[\S2]{GHK:systole}, we have
\[ \kcw(g) \geq \frac{1}{2\pi \lambda_b(\Lambda^2(M\times N))} \geq \frac{\mathrm{stsys}_2(g)}{2\pi \Gamma_b}, \]
as desired.
\end{proof}

\subsection{Line bundle A-hat--cowaist of products of 2--spheres}
In the case of a product of 2--spheres, we can find a lower bound for the line bundle $\ahat$--cowaist.

\begin{lemma}\label{lem:cw-s2}
Let $M = (S^2)^n$ be the product of $n$ copies of $S^2$. Let $g$ be a Riemannian metric on $M$. Then
\[ \acw^{\mathrm{line}}(g) \geq \frac{\mathrm{stsys}_2(g)}{2\pi V_n\Gamma_n}, \]
where
\[ V_n \coloneq \left\lfloor \frac{n+1}{2}\right\rfloor \left\lceil \frac{n+1}{2}\right\rceil. \]
\end{lemma}
\begin{proof}
Let $\Xi_1, \hdots, \Xi_n$ be a linearly independent spanning set for $\Lambda^2(M)$ with
\[ \|\Xi_j\|^* \leq \lambda_n(\Lambda^2(M)). \]
Let $\eta_j$ be the pullback of the generator of $H^2(S^2, \Z)$ by the projection onto the $j$--th 2--sphere. By Proposition \ref{prop:lin-alg}, there are coefficients $a_1, \hdots, a_n \in \Z$ with
\[ \sum_{j=1}^n |a_j| \leq V_n \]
so that
\[ a_1\Xi_i + \hdots + a_n\Xi_n = b_1\eta_1 + \hdots + b_n\eta_n \]
for $b_j \in \Z \setminus \{0\}$. Hence, the top degree term of
\[ (a_1\Xi_i + \hdots + a_n\Xi_n)^n \]
is nonzero.

Fix $\eps > 0$, and let $\omega \in a_1\Xi_i + \hdots + a_n\Xi_n$ satisfy
\[ \|\omega\|_{\infty} \leq \|a_1\Xi_i + \hdots + a_n\Xi_n\|^* + \eps/(2\pi). \]
Let $L \to M$ be the Hermitian line bundle with metric connection associated to $\omega$ from Lemma \ref{lem:prescribe-curvature}. Since $\hat{\mathbf{A}}(M) = 1$, we have
\[ \int_M \hat{\mathbf{A}}(M) \cdot \mathrm{ch}(L) \neq 0, \]
so $L$ is $\ahat$--admissible. Furthermore, we have
\begin{align*}
\|R^L\|_{\infty} = 2\pi\|\omega\|_{\infty}
& \leq 2\pi\|a_1\Xi_i + \hdots + a_n\Xi_n\|^* + \eps\\
& \leq 2\pi\sum_{j=1}^n |a_j|\|\Xi_j\|^* + \eps\\
& \leq 2\pi V_n\lambda_n(\Lambda^2(M)) + \eps.
\end{align*}
Since $\eps > 0$ was arbitary, by \cite[\S2]{GHK:systole}, we have
\[ \acw^{\mathrm{line}}(g) \geq \frac{1}{2\pi V_n \lambda_n(\Lambda^2(M))} \geq \frac{\mathrm{stsys}_2(g)}{2\pi V_n\Gamma_n}, \]
as claimed.
\end{proof}

\subsection{Line bundle A-hat--cowaist of 3--dimensional complex projective space}
Towards Theorem \ref{thm:cp3}, we compute a lower bound for the line bundle $\ahat$--cowaist of $\CP^3$.

\begin{lemma}\label{lem:cw-cp3}
Let $g$ be a Riemannian metric on $\CP^3$. Then
\[ \acw^{\mathrm{line}}(g) \geq \frac{\mathrm{stsys}_2(g)}{4\pi}. \]
\end{lemma}
\begin{proof}
Let $\Xi \in \Lambda^2(M)$ be a generator of $H^2(M, \Z)$. Since $b^1(\CP^3) = 1$, we have
\[ \|\Xi\|^* = \frac{1}{\|\sigma\|} = \frac{1}{\mathrm{stsys}_2(g)}, \]
where $\sigma \in \Lambda_2(M)$ is a generator of $H_2(M, \Z)$.

Fix  $\eps > 0$, and let $\omega \in 2\Xi$ satisfy
\[ \|\omega\|_{\infty} \leq 2\|\Xi\|^* + \eps/(2\pi). \]
Let $L$ be the Hermitian line bundle with metric connection associated to $\omega$ from Lemma \ref{lem:prescribe-curvature}. Then we have
\[ \mathrm{ch}(L) = 1 + 2\Xi + \frac{4\Xi^2}{2} + \frac{8\Xi^3}{6}. \]
Moreover, we have
\[ \hat{\mathbf{A}}(\CP^3) = 1 - \frac{\Xi^2}{6}, \]
so the top degree term of $\hat{\mathbf{A}}(\CP^3) \cdot \mathrm{ch}(L)$ is
\[ \left(-\frac{2}{6} + \frac{8}{6}\right)\Xi^3 \neq 0. \]
Hence, $L$ is $\ahat$--admissible. We also have
\[ \|R^L\|_{\infty} = 2\pi \|\omega\|_{\infty} \leq 4\pi\|\Xi\|^* + \eps = \frac{4\pi}{\mathrm{stsys}_2(g)} + \eps. \]
Since $\eps > 0$ is arbitrary, the conclusion follows.
\end{proof}

\section{Stable 2--systole bounds in positive scalar curvature}\label{sec:main}

Combining the cowaist upper bound in positive scalar curvature from spin geometry with the cowaist lower bounds in terms of the stable 2--systole from the previous section, we obtain the main results of the paper.

\subsection{General non--quantitative bounds}
Theorem \ref{thm:gen} is an immediate consequence of the preceeding work.

\begin{proof}[Proof of Theorem \ref{thm:gen}]
Combine Theorem \ref{thm:spin} with Lemmas \ref{lem:cw-gen} and \ref{lem:cw-gen-prod}.
\end{proof}

We note that Theorem \ref{thm:cpn} is a direct corollary of Theorem \ref{thm:gen} with $M = \CP^{2n+1}$ and $N = \{\mathrm{pt}\}$ (see \cite[Chapter II, Example 2.4]{Lawson-Michelson:spin}).

For Theorems \ref{thm:txs2} and \ref{thm:txs2n}, we have to make an additional argument to bound the stable \emph{spherical} 2--systole instead of the stable 2--systole.

\begin{proof}[Proof of Theorem \ref{thm:txs2} and \ref{thm:txs2n}]
We note that Theorem \ref{thm:txs2} follows from the more general Theorem \ref{thm:txs2n}, which we prove now.

If $n$ is odd, then we can take a Riemannian product with a large circle and the conclusion holds. Hence, we assume $n = 2k$.

Let $\pi_l: ((S^2)^m \times T^{2k}, \tilde{g}) \to ((S^2)^m \times T^{2k}, g)$ be the cover corresponding to the $l$--fold cover of each $S^1$ factor in $T^{2k}$. Let $(\pi_l)_* : \Lambda_2((S^2)^m \times T^{2k}) \to \Lambda_2((S^2)^m \times T^{2k})$ be the restriction of the induced homomorphism on homology.

For $j \in \{1, \hdots, m\}$, let $\sigma_j$ be the homology class of points times the 2--sphere in the $j$--th factor:
\[ \{\mathrm{pt}\} \times \cdots \times S^2 \times \cdots \times \{\mathrm{pt}\} \times \{\mathrm{pt}\}. \]
For $l = \binom{2k}{2}$, let $\tau_1, \hdots, \tau_l$ be the homology classes of points times two circles in the $T^{2k}$ factor. Then $\{\sigma_1, \hdots, \sigma_m, \tau_1, \hdots, \tau_l\}$ is a generating set for $\Lambda_2((S^2)^m \times T^{2k})$. Since the covering space and the base have the same topology, we can use the same basis for both. To avoid ambiguity, we will put a subscript on the stable norm to specify which metric we are using.

We note that
\begin{equation}\label{eqn:cover} (\pi_l)_*(\sigma_j) = \sigma_j,\ \ (\pi_l)_*(\tau_j) = l^2\tau_j. \end{equation}
Since $\pi$ is a local isometry, for any $\sigma \in \Lambda_2((S^2)^m \times T^{2k})$, we have
\begin{equation}\label{eqn:norm-increasing} \|\sigma\|_{\tilde{g}} \geq \|(\pi_l)_*(\sigma)\|_g. \end{equation}
Moreover, since $S^2$ is simply connected, there is a constant $C > 0$ so that
\[ \|\sigma_j\|_{\tilde{g}} \leq C \]
for all $j$ and all $l$. Hence, for $l$ sufficiently large, we have
\[ \mathrm{stsys}_2(\tilde{g}) = \inf\{\|\sigma\|_{\tilde{g}} : \sigma \in \mathrm{Span}(\sigma_1, \hdots, \sigma_j) \setminus \{0\}\}. \]
Applying Theorem \ref{thm:gen} to $\tilde{g}$ (with, for example, $M = (S^2)^m \times T^{2k}$ and $N = \{\mathrm{pt}\}$) and using \eqref{eqn:cover} and \eqref{eqn:norm-increasing}, we have
\begin{align*}
\mathrm{stsys}_2^{\mathrm{sph}}(g)
& = \inf\{\|\sigma\|_g : \sigma \in \mathrm{Span}(\sigma_1, \hdots, \sigma_j) \setminus \{0\}\}\\
& = \inf\{\|(\pi_l)_*(\sigma)\|_g : \sigma \in \mathrm{Span}(\sigma_1, \hdots, \sigma_j) \setminus \{0\}\}\\
& \leq \inf\{\|\sigma\|_{\tilde{g}} : \sigma \in \mathrm{Span}(\sigma_1, \hdots, \sigma_j) \setminus \{0\}\}\\
& = \mathrm{stsys}_2(\tilde{g})\\
& \leq C_{m, n},
\end{align*}
as desired.
\end{proof}

Finally, we prove Theorem \ref{thm:non-2-ess}.

\begin{proof}[Proof of Theorem \ref{thm:non-2-ess}]
Since $X$ is a homology 3--sphere, $X$ is orientable, as we have
\[ H_3(X,\Z) \cong H_3(S^3, \Z) \cong \Z. \]
Since $X$ is a closed oriented 3--manifold, $X$ is spin (see \cite[Chapter II, Example 2.3]{Lawson-Michelson:spin}). Therefore $S^2 \times X \times X$ is spin. Since $X$ is hyperbolic, $X \times X$ admits a metric of nonpositive curvature (take the product of hyperbolic metrics). Then by \cite{Gromov-Lawson:closed} (see \cite[Chapter IV, Theorem 5.4(B)]{Lawson-Michelson:spin}), $X \times X$ is enlargeable. Hence, the result follows from Theorem \ref{thm:gen} with $M = S^2$ and $N = X \times X$.
\end{proof}

\subsection{Non--quantitative bound for the volume of K{\"a}hler metrics on complex projective spaces}
We make a digression to prove the curious Theorem \ref{thm:cpn-kahler}.

\begin{proof}[Proof of Theorem \ref{thm:cpn-kahler}]
Let $\omega$ be the K{\"a}hler form corresponding to the K{\"a}hler metric $g$. Then $[\omega] = A\Xi$, where $\Xi$ is the generator of $H^2(\CP^n, \Z)$. Inspecting the proof of Lemma \ref{lem:cw-gen}, we find that
\[ \kcw(g) \geq \frac{1}{2\pi\|\omega/A\|_{\infty}} = \frac{|A|}{2\pi}, \]
where the last equality follows from the Wirtinger formula. Moreover, the Wirtinger formula implies
\[ \mathrm{Vol}(\CP^n, g) = \frac{1}{n!}\left|\int_{\CP^n} \omega^n\right| = \frac{|A|^n}{n!}\left|\int_{\CP^n} \Xi^n\right|= \frac{|A|^n}{n!}. \]
Hence, we have
\[ \kcw(g) \geq \frac{(n!\mathrm{Vol}(\CP^n, g))^{1/n}}{2\pi}. \]
The theorem follows by Theorem \ref{thm:spin} and \cite[Chapter II, Example 2.4]{Lawson-Michelson:spin}.
\end{proof}

\subsection{Quantitative bound for products of 2--spheres}
We use the quantitative upper bound for the line bundle $\ahat$--cowaist from Theorem \ref{thm:spin-quant} together with the lower bound by the stable 2--systole for products of 2--spheres from Lemma \ref{lem:cw-s2}.

\begin{proof}[Proof of Theorem \ref{thm:s2xs2}]
Since $V_2 = 2$ (by direct computation) and $\Gamma_2 = \frac{3}{2}$ (see \cite[Proposition 2.3]{GHK:systole}), Lemma \ref{lem:cw-s2} implies
\[ \acw^{\mathrm{line}}(g) \geq \frac{\mathrm{stsys}_2(g)}{6\pi}. \]
Since $S^2 \times S^2$ is spin, Theorem \ref{thm:spin-quant} and $\mathrm{scal}(g) \geq 4$ implies
\[ \acw^{\mathrm{line}}(g) \leq 6. \]
The result follows.
\end{proof}

\begin{proof}[Proof of Theorem \ref{thm:s2n}]
Since $b_2((S^2)^n) = n$, Lemma \ref{lem:cw-s2} implies
\[ \acw^{\mathrm{line}}(g) \geq \frac{\mathrm{stsys}_2(g)}{2\pi V_n \Gamma_n}. \]
Since $(S^2)^n$ is spin, Theorem \ref{thm:spin-quant} and $\mathrm{scal}(g) \geq 2n$ implies
\[ \acw(g) \leq 2(2n-1). \]
The result follows from the fact that $V_n = O(n^2)$ (directly from the formula) and $\Gamma_n = O(n\log n)$ by \cite{Banaszczyk:gamma} (see \cite[\S2]{GHK:systole}).
\end{proof}

\subsection{Quantitative bound for 3--dimensional complex projective space}
Finally, we use the quantitative bounds from Theorem \ref{thm:spin-quant} and Lemma \ref{lem:cw-cp3} to obtain explicit bounds for $\CP^3$.

\begin{proof}[Proof of Theorem \ref{thm:cp3}]
Since $\CP^3$ is spin and $\mathrm{scal}(g) \geq 48$, Theorem \ref{thm:spin-quant} implies
\begin{equation}\label{eqn:acw-cp3} \acw^{\mathrm{line}}(g) \leq \frac{5}{4}. \end{equation}

By Lemma \ref{lem:cw-cp3}, we have
\[ \acw^{\mathrm{line}}(g) \geq \frac{\mathrm{stsys}_2(g)}{4\pi}. \]
The stable 2--systole bound follows.

For the volume bound, let $\omega$ be the K{\"a}hler form corresponding to the K{\"a}hler metric $g$. Then $[\omega] = A\Xi$, where $\Xi$ is the generator of $H^2(\CP^3, \Z)$. Inspecting the proof of Lemma \ref{lem:cw-cp3}, we find that
\[ \acw^{\mathrm{line}}(g) \geq \frac{1}{2\pi\|2\omega/A\|_{\infty}} = \frac{|A|}{4\pi}, \]
where the last equality follows from the Wirtinger formula. Moreover, the Wirtinger formula implies
\[ \mathrm{Vol}(\CP^3, g) = \frac{1}{6}\left|\int_{\CP^3} \omega^3 \right| = \frac{|A|^3}{6}\left|\int_{\CP^3} \Xi^3\right| = \frac{|A|^3}{6}. \]
Hence, we have
\[ \acw^{\mathrm{line}}(g) \geq \frac{(6\mathrm{Vol}(\CP^3, g))^{1/3}}{4\pi}, \]
and the conclusion follows from \eqref{eqn:acw-cp3}.
\end{proof}

\section{Appendix: Linear algebra fact}

\begin{proposition}\label{prop:lin-alg}
Let $v_1, \hdots, v_n \in \Z^n$ be a linearly independent set in $\Z^n \subset \R^n$. There are coefficients $a_1, \hdots, a_n \in \Z$ with
\[ \sum_{j=1}^n |a_j| \leq \left\lfloor \frac{n+1}{2}\right\rfloor \left\lceil \frac{n+1}{2}\right\rceil \]
so that
\[ a_1v_1 + \hdots + a_n v_n = (b_1, \hdots, b_n)^T \in \R^n \]
with $b_j \in \Z \setminus \{0\}$.
\end{proposition}
\begin{proof}
We argue by induction on $n$. For $n = 1$, the conclusion is trivial.

Suppose the conclusion holds for $n-1$, and consider the case of $n$. By linear independence, we have (without loss of generality) that the projection onto the first $n-1$ coordinates of $v_1, \hdots, v_{n-1}$ is a linearly independent set in $\Z^{n-1}$. Hence, there are $a_1, \hdots, a_{n-1} \in \Z$ so that
\[ \sum_{j=1}^{n-1} |a_j| \leq \left\lfloor \frac{n}{2}\right\rfloor \left\lceil \frac{n}{2}\right\rceil \]
and
\[ v\coloneq a_1v_1 + \hdots + a_{n-1}v_{n-1} \]
is nonzero in the first $n-1$ coordinates. If the $n$--th coordinate of $v$ is nonzero, then we are done. Otherwise, choose $l \in \{1, \hdots, n\}$ so that the $n$--th coordinate of $v_l$ is nonzero. Note that the $j$--th coordinate of $v + bv_n$ is zero for at most one choice of $b$. Then there is some
\[ c \in \left\{-\left\lceil \frac{n}{2}\right\rceil, \hdots, -1, 1, \hdots, \left\lceil \frac{n}{2}\right\rceil\right\} \]
so that
\[ v' \coloneq a_1v_1 + \hdots + a_{n-1}v_{n-1} + cv_l = a_1'v_1 + \hdots + a_n'v_n \]
is nonzero in every coordinate. Note that $v'$ has the desired form with
\[ \sum_{j=1}^n |a_j'| \leq |c| + \sum_{j=1}^{n-1} |a_j| \leq \left\lfloor \frac{n}{2}\right\rfloor\left\lceil \frac{n}{2} \right\rceil + \left\lceil \frac{n}{2} \right\rceil. \]
One easily checks by cases depending on the parity of $n$ that the right hand side above equals the desired bound.
\end{proof}

 % REFERENCES
 \bibliographystyle{amsalpha}
 \bibliography{spin_and_systole}

\end{document}